\documentclass[10pt,psamsfonts]{amsart}
\usepackage{amsmath}
\usepackage{amsthm}
\usepackage{amssymb}
\usepackage{amscd}
\usepackage{amsfonts}
\usepackage{amsbsy}
\usepackage{color}
\usepackage{graphicx,psfrag}
\usepackage{graphics}
\usepackage{lipsum}

\newcommand{\A}{\ensuremath{\mathcal{A}}}
\newcommand{\Ii}{\ensuremath{\mathcal{I}}}

\newcommand{\E}{\ensuremath{\mathcal{E}}}
\newcommand{\R}{\ensuremath{\mathbb{R}}}

\newcommand{\C}{\ensuremath{\mathbb{C}}}

\newcommand{\sss}{\ensuremath{\mathbb{S}}}

\newcommand{\X}{\mathcal{X}}
\newcommand{\Y}{\mathcal{Y}}

\def\p{\partial}

\newtheorem {theorem} {Theorem}
\newtheorem {proposition} [theorem]{Proposition}

\newtheorem {lemma}  [theorem]{Lemma}

\newtheorem {remark} [theorem]{Remark}

\begin{document}

\title[Darboux theory of integrability on $\mathbb{E}^n$]
{Darboux theory of integrability for real polynomial vector fields on the $n-$dimensional ellipsoid}

\author[J. Llibre and A.C. Murza]
{Jaume Llibre and Adrian C. Murza}



\begin{abstract}
We extend to the $n$-dimensional ellipsoid contained in $\R^{n+1},$ the Darboux theory of integrability for polynomial vector fields in the $n$-dimensional sphere (Llibre et al., 2018). New results on the maximum number of invariant parallels and meridians of polynomial vector fields $\X$ on the invariant $n-$dimensional ellipsoid, as a function of its degree, are provided. Our results extend the known result on the upper bound for the number of invariant hyperplanes that a polynomial vector field $\Y$ in $\R^n$
can have in function of the degree of $\Y$.
\end{abstract}

\maketitle

\section{Introduction and statement of the main results}

Real nonlinear vector fields defined on ellipsoids appear in many applied sciences, from chemistry \cite{JPC, Science2, Science3, Naturex, JPCA, JPCA4, JPCx, Naturex3} and robotics \cite{Automatica, IFAC, Autom, King} to physics \cite{JMPhys, Arch, ND, MDPI, Nature1, Nature2, Nature3, Science1} and geodesy \cite{Geod}. They are used to model a wide variety of processes like liquid crystals \cite{Naturex3, Science3, Science2}, optics \cite{Science1, Naturex2, Science4}, single and collective oscillations of Rydberg atoms \cite{Nature1}, molecular collision complexes \cite{JPC}, the role of initial rotational states chemical reaction reactions \cite{JPCA4}, vortex dynamics \cite{JMPhys, Science1} to cite only a few of them from the recent years. These applications caught the attention of theoretical chemists, physicists and mathematicians in search for a better understanding of the underlying phenomena \cite{JMPhys, JPCx, ND, Arch, MDPI, JMAA, Geod, Tsi}. Key concepts in the invariant theory of dynamical systems such as first--integrals and functions invariant by the flow are of great interest and importance in these studies.

First--integrals of $n-$dimensional real vector fields are important for two main reasons. They allow characterizing the phase portraits of these systems and reduce their dimension. A differential system in $\R^n$ is said to be {\it integrable} if it has $n-1$ linearly independent first--integrals. Therefore, methods to establish the existence of independent firs--integrals, their number and explicit form are very important in the qualitative theory of differential equations and implicitly in applied sciences.

Among the techniques of finding such conserved quantities, Hamiltonians are the first--integrals easiest to detect. For non-Hamiltonian vector fields other techniques have been developed, such as the Noether symmetries \cite{CS}, the Darboux theory of integrability \cite{Da}, the Lie symmetries \cite{Ol}, the Painlev\'e analysis \cite{BGR}, the use of Lax pairs \cite{Lax}, or the direct method \cite{GRZ} and \cite{Hi}, to cite only some of them.

In this paper we are interested in the Darboux theory of integrability for real polynomial vector fields in $\R^n,$ applied to the $n-$dimensional ellipsoids. The original theory published in 1878, uses the number of invariant algebraic hypersurfaces of polynomial vector fields to construct first--integrals. During the last few decades it originated intense studies, being extended and/or refined by many authors in $\R^2$ see \cite{CLS, Ch, CL1, CL2, Da, Jo, Ll, LZ3, Pe, Po, PS, Sc1, Sc2, Sc3, Si}, and in $\R^n$ see \cite{LB, LM, LZ, LZ1, LZ2, LZ4}.

The first objective of this paper is to present the Darboux theory of integrability of the real polynomial vector fields on the $n$-dimensional ellipsoid $\mathbb{E}^n= \{x\in \R^{n+1} : \sum_{i=1}^{n+1}x_i^2/a_i^2=1\},$  and to study the maximum number of invariant parallels and meridians that such vector fields can have. This represents in some sense the natural extension to $\mathbb{E}^n$ of the results on the $n-$dimensional sphere \cite{LM1} and directly provide algebraic bounds and explicit expressions for the varieties invariant by the vector field defined on the $n-$dimensional ellipsoid.

Our second objective is to show the parallelism between the Darboux theory of integrability of the real polynomial vector fields on
$\R^n$ and on $\mathbb{E}^n.$ This includes the parallelism between the invariant hyperplanes of the polynomial vector fields in $\R^n$ with the ones on $\mathbb{E}^n,$ and uses it to provide bounds and explicit expressions for the invariant parallels and meridians on $\mathbb{E}^n.$

\subsection{Darboux theory of integrability in $\R^n$ and invariant hyperplanes}

Let us consider the polynomial vector field
\begin{equation}\label{X}
\X=\displaystyle{\sum\limits_{i=1}\limits^{n}P_i(x_1,\ldots,x_{n})\frac{\p}{\p x_i},\qquad (x_1,\ldots,x_{n})\in \mathbb R^{n}},
\end{equation}
in $\R^{n}$, where $P_i$ for $i=1,\ldots,n$ are polynomials of degree at most $m$ and $m_i$ denotes the degree of the polynomial $P_i$. Then we say that $m=(m_1,\ldots,m_n)$ is the {\it degree} of the polynomial vector field $\X$. With no loss of generality for the rest of the paper we assume $m_1\geqslant\ldots\geqslant m_n.$ 

Let us deote by $\C[x_1,\ldots,x_n]$ the set of all polynomials in the variables $x_1,\ldots,x_n$ with complex coefficients and by $\R[x_1,\ldots,x_n]$ the ones with real coefficients. We use the term {\it invariant algebraic
hypersurface} of the polynomial vector field $\X$ to denote the
algebraic hypersurface $f=f(x_1,\ldots,x_d)=0$ with $f\in
\C[x_1,\ldots,x_n]$ such that for some polynomial $K\in
\C[x_1,\ldots,x_n],$ we have $\X f=Kf$. The polynomial
$K$ is called the {\it cofactor} of the invariant algebraic
hypersurface $f=0.$ If the degree of
the polynomial vector field $\X$ is ${\textbf m}=(m_1,\ldots,m_n),$ with
$m_1\ge m_2 \ge \ldots \ge m_n$, then the degree of any cofactor is at most
degree $m_1-1$. The algebraic hypersurface $f=0$ is called
an {\it invariant hyperplane} if the degree of $f$ is $1$.

The invariant algebraic hypersurfaces are used to construct the first--integrals in the Darboux theory of integrability. They separate the
phase space of the polynomial vector field $\X$ into invariant
pieces, which makes easier studying the dynamics of the vector field $\X.$

As mentioned previously, the polynomial vector field $\X$ is real. Nevertheless, we work with
the complex invariant algebraic hypersurfaces of $\X$, which sometimes may be real. After stating the Theorem \ref{t1} we will
explain in detail this particular situation. For now, we just mention that
in some cases knowing the explicit form of the
complex invariant algebraic hypersurfaces may help finding the real first integrals of a real
polynomial vector field $\X.$ 

One can express the invariance of $f$ under the vector field, expressed as follows. If the orbit of the
polynomial vector field $\X$ intersects at a point the algebraic
hypersurface $f=0,$ then the entire orbit is contained in
$f=0$. In such case we say that the algebraic hypersurface $f=0$ is {\it invariant} by the vector field of $\X$ because it is
invariant by the flow of $\X$. In \cite{LM} it is proved also the converse.

The last 40 years have brought important refinements of the Darboux theory of integrability. Among them, one of the most important is the exponential factors, introduced by
Christopher \cite{Ch}. The {\it
exponential factors} $F(x_1,\ldots,x_n)$ of the polynomial vector
field $\X$ of degree ${\textbf m}$ are defined as exponential functions of
the form $F=\exp(g/h)$ with $g$ and $h$ polynomials in
$\C[x_1,\ldots,x_{n}]$ satisfying $\X F=LF$ for some $L\in
\C_{m_1-1}[x_1,\ldots,x_n],$ where $\C_{m_1-1}[x_1,\ldots,x_n]$
denotes the set of all polynomials of $\C[x_1,\ldots,x_n]$ of degree
at most $m_1-1$.

The relationship of the exponential factors with the exponential factors with the multiplicity
of the invariant algebraic hypersurface $h=0$ if $h$ is not a
constant, is well known. The case $h=\mathrm{constant}$ is related to the multiplicity of
the infinity, as analyzed in \cite{CLP, LZ1, LZ4}. 

A multiplicity $k$ of the invariant algebraic hypersurface $h=0$ can be interpreted as follows.
Let us consider a class $\mathcal{C}$ of polynomial vector fields of the same degree. A small perturbation to the vector field $\X$ in $\mathcal{C}$ gives rise to
polynomial vector fields $\Y$ near $\X$ having $k$ different
invariant algebraic hypersurfaces. These have the property that when $\Y \to \X$ these $k$ hypersurfaces approach to the hypersurface $h=0$.

Let us consider an open set $\Omega\subset \R^{n}.$ Then a real function
$\Ii(x_1,\ldots,x_{n},t): \R^{n} \times \R \to \mathbb R$, is called an {\it
invariant} of the polynomial vector field $\X$ on $\Omega$, if
$\Ii(x_1(t),\ldots,x_{n}(t),t)$ is constant for any $t$
such that the orbit $(x_1(t),\ldots,x_{n}(t))$ of $\X$ is contained in $\Omega$.

A time-independent invariant function $\Ii$ is called a {\it first
integral}. If $\Ii$ is a rational function in its
variables, then we say it is a {\it rational first integral} of the $\X$ on $\Omega$.

Next we summarize the Darboux theory for polynomial
vector fields in $\R^n,$ which we we use afterwards to extend to the Darboux theory for polynomial vector fields in $\mathbb{E}^n$.

\begin{theorem}\label{t1}
Suppose that a polynomial vector field $\X$ defined in $\mathbb R^n$
of degree $m= (m_1,\ldots,m_n)$ admits $p$ invariant algebraic
hypersurfaces $f_i=0$ with cofactors $K_i$ for $i=1,\ldots,p$, and
$q$ exponential factors $F_j=\exp(g_j/h_j)$ with cofactors $L_j$ for
$j=1,\ldots,q$. Then the following statements hold.
\begin{itemize}
\item[$(a)$] There exist $\lambda_i,\mu_j\in \mathbb C$ not all
zero such that $\displaystyle{\sum\limits_{i=1}\limits^{p}\lambda_iK_i+
\sum\limits_{j=1}\limits^{q} \mu_jL_j=0}$, if and only if the real
$($multi-valued$)$ function of Darbouxian type
\begin{equation*}\label{e1}
f_1^{\lambda_1}\ldots f_p^{\lambda_p}F_1^{\mu_1}\ldots F_q^{\mu_q},
\end{equation*}
substituting $f_i^{\lambda_i}$ by $\left|f_i\right|^{\lambda_i}$ if
$\lambda_i\in \mathbb R$, is a first integral of the vector field
$\X$.

\smallskip

\item[$(b)$] If $p+q\geqslant \left(\begin{array}{c}n+m_1-1\\m_1-1\end{array}
\right)+1$, then there exist $\lambda_i,\mu_j\in \mathbb C$ not all
zero such that $\displaystyle{\sum\limits_{i=1}\limits^{p}\lambda_iK_i+
\sum\limits_{j=1} \limits^{q}\mu_jL_j=0}$.

\smallskip

\item[$(c)$] There exist $\lambda_i,\mu_j\in \mathbb C$ not all zero
such that $\displaystyle{\sum\limits_{i=1}\limits^{p}\lambda_iK_i+
\sum\limits_{j=1} \limits^{q}\mu_jL_j=-\sigma}$ for some $\sigma\in
\mathbb R\backslash\{0\}$, if and only if the real
$($multi-valued$)$ function
\[
f_1^{\lambda_1}\ldots f_p^{\lambda_p}F_1^{\mu_1}\ldots
F_q^{\mu_q}e^{\sigma t},
\]
substituting $f_i^{\lambda_i}$ by $\left|f_i\right|^{\lambda_i}$ if
$\lambda_i\in\mathbb R$, is an invariant of the vector field $\X$.

\smallskip

\item[$(d)$] The vector field $\X$ has a rational first integral if
and only if
\[
p+q\geqslant\left(\begin{array}{c}n+m_1-1\\m_1-1\end{array}\right)+n.
\]
Moreover, all trajectories are contained in invariant algebraic
hypersurfaces.
\end{itemize}
\end{theorem}

The proof of Theorem \ref{t1} is given in \cite{DLA, LZ1,
LZ2, LZ3}.

In Theorem \ref{t1} we only
have considered partially the multiplicity of the invariant
algebraic hypersurfaces $f_i=0$ from the viewpoint of the existence of exponential factors, as discussed \cite{CLP} and \cite{LZ1}.

According to the statement (a) of Theorem \ref{t1}, function
\eqref{e1} is real. This fact results from the following deduction. Recall that he vector field $\X$ is real. It is well known that complex invariant algebraic hypersurface or exponential factor come in conjugate pairs, and they appear simultaneously. If such an complex conjugate pair $f=0$ and
$\overline f=0$ of invariant
algebraic hypersurfaces of $\X$ occur, then among the factors of the first integral \eqref{e1} there is a real one of the form $f^{\lambda}{\bar f}^{\,\bar\lambda}.$ It is
the multi-valued real function
\[
\left[\left(\mbox{Re}\,f\right)^2+\left(\mbox{Im}\,
f\right)^2\right]^{\mbox{Re}\,
\lambda}\,\mbox{exp}\left(-2\,\mbox{Im}\,\lambda\,\mbox{arctan}\,
\left(\frac{\mbox{Im}\,f}{\mbox{Re}\, f}\right)\right),
\]
with $\mbox{Im}\,\lambda\,\mbox{Im}\,f\not\equiv 0$. Occurence of a complex conjugate pair
$F=\mbox{exp}(h/g)$ and $\overline F=\mbox{exp}(\overline
h/\overline g)$ -among the
exponential factors of $\X$- produces a first--integral \eqref{e1} with a real
factor as
\[
\left(\mbox{exp}\left(\frac{h}{g}\right)\right)^{\mu}\left(\mbox{exp}
\left(\frac{\overline h}{\overline g}\right)\right)^{\overline
\mu}=\mbox{exp}\left(2\,\mbox{Re}\,\left(\mu\,\frac{h}{g}\right)\right).
\]

An important tool in the study of invariant algebraic
hypersurfaces is the {\it extactic polynomial of $\X$ associated to
$W$}. It was introduced by Lagutinskii, see \cite{SW}. 
Let $W$ be a finitely
generated vector subspace of the vector space $\C[x_1,\ldots,x_d].$ A a basis of $W$ is $\{v_{1}$, \ldots, $v_{l}\},$ $l$ is its dimension
and $\X^{j}(v_{i})= \X^{j-1}(\X(v_{i})).$ Then the
{\it extactic polynomial of $\X$ associated to $W$} is
\begin{equation*}\label{21}
{\mathcal E}_W={\mathcal E}_{\{v_{1}, \ldots, v_{l}\}}(\X) =
\det\left(
\begin{array}{cccc}
v_{1} & v_{2} & \ldots & v_{l} \\
\X(v_{1}) & \X(v_{2}) & \ldots & \X(v_{l}) \\
\vdots & \vdots & \ldots & \vdots \\
\X^{l-1}(v_{1}) & \X^{l-1}(v_{2}) & \ldots & \X^{l-1}(v_{l})
\end{array}
\right)=0.
\end{equation*}
An important property of ${\mathcal E}_W$ resides on the fact that defining it is independent on $\{v_{1}$, \ldots, $v_{l}\}.$

We have two reasons for using ${\mathcal
E}_{W}(\X)$ in the study of polynomial vector fields on the $n-$dimensional ellipsoid. In first place, we use it to detect those algebraic
hypersurfaces $f=0$ with $f\in W$ which also satisfy the invariance condition under $\X.$ (For more details see the next proposition
proved in \cite{CLP}). Moreover, it is helpful to compute the multiplicity of the invariant algebraic hypersurfaces.

The following result is crucial in finding the invariant algebraic hypersurfaces of $\X.$ Even though it had been formulated for complex polynomial vector
fields \cite{CLP}, we can still use it
because our polynomial vector fields are real, therefore complex.

\begin{proposition}\label{p1}
Let $\X$ be a polynomial vector field in $\C^n$ and let $W$ be a
finitely generated vector subspace of $\C[x_1,\ldots,x_n]$ with
$\dim (W)>1 $. Then every algebraic invariant hypersurface $f=0$ for
the vector field $\X$ with $f\in W$, is a factor of the polynomial
${\mathcal E}_{W}(\X)$.
\end{proposition}

An immediate consequence of Proposition \ref{p1} is that $f=0$ is an invariant
hyperplane of $\X$ if $f$
is a factor of the polynomial ${\mathcal E}_{W}(\X).$ In \cite{CLP} it has been proved that {\it the
invariant hyperplane $f=0$ has multiplicity $k$} if $k$ is the
greatest positive integer satisfying that $f^k$ is a divisor of the polynomial
${\mathcal E}_{W}(\X)$.

The next result gives bounds for the number of invariant hyperplanes
of a polynomial vector field in $\R^n$ as a function of its degree. The proof can be found in
\cite{LM}. A similar result for polynomial vector fields in
$\R^2$ had been proved in \cite{AGL}.

\begin{theorem}\label{t2}
Assume that a polynomial vector field $\X$ in $\R^n$ or $\C^n$ with
$n\ge 2$ of degree $m= (m_1,\ldots,m_d)$ with $m_1\ge \ldots
\ge m_n$ has finitely many invariant hyperplanes. Then
\begin{enumerate}
\item{}\label{part1} the number of
invariant hyperplanes of $\X$ taking into account their
multiplicities is at most
\begin{equation*}\label{3}
\displaystyle{\left(\begin{array}{c} n\\
2 \end{array} \right)(m_1-1)+\left( \sum_{k=1}^n m_k\right)};
\end{equation*}
\item{}\label{part2} the number of
invariant hyperplanes of $\X$ through a single point taking into account their
multiplicities taking into account their
multiplicities is at most
\begin{equation*}\label{31}
\displaystyle{\left(\begin{array}{c} n-1\\
2 \end{array} \right)(m_1-1)+\left( \sum_{k=1}^{n-1} m_k\right)+1}.
\end{equation*}
\end{enumerate}
Moreover there are complex polynomial vector fields which reach this
upper bound for their number of invariant hyperplanes taking into
account their multiplicities.
\end{theorem}

\begin{remark}
Based on a sufficient number of invariant hyperplanes whose cofactors satisfy either statement (a) or
(b) of Theorem \ref{t1}, one may construct a first--integral of
the polynomial vector field $\X$.
\end{remark}

\subsection{Darboux theory of integrability in $\mathbb{E}^n$ and invariant parallels and meridians}

Let $M:\R^{n+1}\to\R$ be a $C^1$ map. A hypersurface
${\bf\Psi}=\{(x_1,\ldots,x_{n+1})$ $\in \R^{n+1}:
M(x_1,\ldots,x_{n+1})=0\}$ is called {\it regular} if $\nabla M=\left(\frac{\partial}{\partial x_1},\ldots, \frac{\partial}{\partial x_n}\right)\neq0$ on $\bf \Psi$. Clearly, $\Psi$ is smooth if it is regular. Moreover, $\Psi$ is said to be {\it algebraic} if $M$ is a polynomial and algebraic of {\it degree} $d$ if $M$ is a polynomial of degree $d.$ A {\it polynomial vector field}
$\X$ {\it on a regular hypersurface} $\Psi$ is a polynomial vector field
$\X$ in $\R^{n+1}$ in \eqref{X} such that
\begin{equation}\label{on}
(P_1,\ldots,P_{n+1})\cdot\nabla M=0\,\, \mbox{ on }\Psi,
\end{equation}
where $(P_1,\ldots,P_{n+1})$ and $\nabla M$ are vectors in $\R^{n+1}$ and the dot denotes their inner product. If the polynomial vector field $\X$ in $\R^{n+1}$ has
degree ${\textbf m}$, then $\X$ on $\Psi$ is said to have
degree ${\textbf m}$.

Let $f=f(x_1,\ldots,x_{n+1})\in \mathbb C[x_1,\ldots,x_{n+1}]$. The algebraic hypersurface
$\{f=0\}\cap\Psi\subset\R^{n+1}$ is said to be {\it invariant by the
polynomial vector field} $\X$ {\it on} $\Psi$ (or simply an {\it
invariant algebraic hypersurface on } $\Psi$) if the following two conditions hold:
\begin{itemize}
\item[(i)] there exists a polynomial $k\in\mathbb C[x_1,\ldots,x_{n+1}]$ called the {\it cofactor} of $f=0$ on
$\Psi,$
such that
\[
\displaystyle{\X f=\sum\limits_{i=1}\limits^{n+1}P_i\frac{\p f}{\p x_{i}}=kf\,\,
\mbox{ on }\Psi},
\]

\smallskip

\item[(ii)] hypersurfaces $f=0$ and $\Psi$ intersect transversely
i.e. $\nabla M$ and $\nabla f$ are
linearly independent on $\{f=0\}\cap\Psi$.
\end{itemize}

The algebraic hypersurface
$\{f=0\}\cap\Psi$ is invariant by the flow of the vector field
$\X$. A simple expanation of this fact is that $\X$ is tangent to the algebraic
hypersurface $\{f=0\}\cap\Psi,$ and therefore the hypersurface
$\{f=0\}\cap\Psi$ is formed by orbits of $\X$.

An exponential function of the form
$F=\exp(g/h)$ of degree $m$ on the regular
hypersurface $\Psi$ and satisfying $\X F=LF$ on $\Psi$ for some $K\in
\C_{m-1}[x_1,\ldots,x_{n+1}],$ is called {\it exponential factor} $F=F(x_1,\ldots,x_{n+1})$ of $\X,$ where $g$ and $h$ polynomials in $\C[x_1,\ldots,x_{n+1}].$

Let $\Psi=\{M=0\}$ be a smooth algebraic hypersurface in
$\R^{n+1}$ of degree $d$. Consider the set $\A=\C_m[x_1,x_2,\ldots,x_{n+1}]$ and let $\sim$ be the equivalence relation on $\A,$
$$\displaystyle{\left\{f\in\C_m[x_1,x_2,\ldots,x_{n+1}]|f\sim g~\mathrm{if}~f/g= \mbox{constant}~\mathrm{or}~f-g=hM\right\}},$$
where $f, g$ are two polynomials in $\A.$  
Then $\Psi\cap \{f=0\}=
\Psi\cap \{g=0\}$. The result of the partition of
$\A$ into equivalence classes yields the quotient space
$\A/\sim;$ we denote its dimension by $d(m)$,
called the {\it dimension of} $\C_m[x_1,\ldots,x_{n+1}]$ {\it on}
$\Psi$. It has been proved in \cite{LZ} that the dimension of $\A/\sim$ is
\[
d(m)=\left(\begin{array}{c}n+m\\n\end{array}\right)-
\left(\begin{array}{c}n+m-d\\n\end{array}\right).
\]

Let $\Omega\subset \R^{n+1}$ be an open set. A real function
$\Ii(x_1,\ldots,x_{n+1},t): \R^{n+1} \times \R \to \mathbb R$, is said
to be {\it invariant} by $\X$ on
$\Psi\cap \Omega$, if $\Ii=\mbox{constant}~\forall t|(x_1(t),\ldots,x_{n+1}(t))\subset\Psi\cap \Omega$.

If an $\X-$invariant $\Ii$ is independent on $t$, then $\Ii$ is a {\it first
integral}. If a first integral $\Ii$ is a rational function, then it is called a {\it rational first integral} of the
vector field $\X$ on $\Psi\cap \Omega$.

In the following we present the extension of the Darboux theory of integrability to polynomial vector fields on $\mathbb{E}^n$.

\begin{theorem}\label{t3}
Assume that $\X$ is a polynomial vector field on $\mathbb{E}^n$ of degree
${\bf m}=(m_1,\ldots,m_{n+1})$ having $p$ invariant algebraic hypersurfaces
$\{f_i=0\}\cap\mathbb{E}^n$ with cofactors $K_i$ for $i=1,\cdots,p$, and
$q$ exponential factors $F_j=\exp(g_j/h_j)$ with cofactors $L_j$ for
$j=1,\cdots,q$. Then the following statements hold.
\begin{itemize}
\item[(a)] There exist $\lambda_i,\mu_j\in \mathbb C$ not all
zero such that $\displaystyle{\sum\limits_{i=1}\limits^{p}\lambda_iK_i+
\sum\limits_{j=1}\limits^{q} \mu_jL_j=0}$ on $\mathbb{E}^n$,  if and only
if the real $($multi--valued$)$ function of Darbouxian type
\begin{equation*}
f_1^{\lambda_1}\cdots f_p^{\lambda_p}F_1^{\mu_1}\cdots F_q^{\mu_q},
\label{a1}
\end{equation*}
substituting $f_i^{\lambda_i}$ by $\left|f_i\right|^{\lambda_i}$ if
$\lambda_i\in \mathbb R$, is a first integral of the vector field
$\X$ on $\mathbb{E}^n$.

\smallskip

\item[(b)] If $p+q\geqslant \frac{\displaystyle n+2m_1}{\displaystyle
n+m_1}\left(\begin{array}{c}n+m_1\\m_1\end{array}\right)+1$, then
there exist $\lambda_i,\mu_j \in \mathbb C$ not all zero such that
$\displaystyle{\sum\limits_{i=1}\limits^{p}
\lambda_iK_i+\sum\limits_{j=1}\limits^{q}\mu_jL_j=0}$ on $\mathbb{E}^n$.

\smallskip

\item[(c)] There exist $\lambda_i,\mu_j\in \mathbb C$ not all zero
such that $\displaystyle{\sum\limits_{i=1}\limits^{p}\lambda_iK_i+
\sum\limits_{j=1}\limits^{q} \mu_jL_j=-\sigma}$ on $\mathbb{E}^n$ for some
$\sigma\in\mathbb R\backslash\{0\}$, if and only if the real
$($multi--valued$)$ function
\[
f_1^{\lambda_1}\cdots f_p^{\lambda_p}F_1^{\mu_1}\cdots
F_q^{\mu_q}e^{\sigma t},
\]
substituting $f_i^{\lambda_i}$ by $\left|f_i\right|^{\lambda_i}$ if
$\lambda_i\in\mathbb R$, is an invariant of $\X$ on $\mathbb{E}^n$.

\smallskip

\item[(d)] The vector field $\X$ on $\mathbb{E}^n$ has a rational first
integral if and only if $p+q\geqslant\frac{\displaystyle
n+2m_1}{\displaystyle n+m_1}\left(\begin{array}{c}n+
m_1\\m_1\end{array}\right)+n$. Moreover, all trajectories are
contained in invariant algebraic hypersurfaces.
\end{itemize}
\end{theorem}

Statements (a), (b) and (c) of Theorem \ref{t3} have been proved in \cite{LZ}, while for the proof of statement (d) see \cite{LB}.

The {\it parallels} of the $n$-dimensional ellipsoid $\mathbb{E}^n$ are closed curves that result from the
intersection of the hyperplanes $x_{n+1}=$ constant with the ellipsoid
$\mathbb{E}^n$ giving a topological circle. The {\it meridians} are closed curves (also topologhical circles) that result from the intersections of the
hyperplanes containing the $x_{n+1}$--axis with the ellipsoid $\mathbb{E}^n.$
In the following we extend Theorem \ref{t2} to the parallels and
meridians.

Next theorem provides the maximum number of invariant
meridians that a polynomial vector field in $\mathbb{E}^n$ can have in function of its degree.

\begin{theorem}\label{t4}
For $n\geqslant 2$ let $\mathcal{X}$ be a polynomial vector field on
$\mathbb{E}^n$ of degree ${\bf m}=(m_1,\ldots,m_{n+1})$ with $m_1\geqslant
m_2\geqslant\ldots\geqslant m_{n+1}.$ Assume that $\mathcal{X}$ has
finitely many invariant meridians. Then the number of invariant
meridians of $\mathcal{X}$ is at most
\begin{equation}\label{mer}
\displaystyle \binom{n-1}{2}
(m_1-1)+ \left(\sum_{i=1}^{n-1} m_i\right)+1,
\end{equation}
 where $\displaystyle
\binom{n-1}{2}=0$ if $n=2$.
\end{theorem}

\begin{proof}
The same as the proof of Theorem 6 in \cite{LM1}.
\end{proof}

In the next theorem we provide the maximum number of invariant
parallels that a polynomial vector field in $\mathbb{E}^n$ can have in
function of its degree.

\begin{theorem}\label{t5}
Let $\mathcal{X}$ be a polynomial vector field on $\mathbb{E}^n$ of degree
${\bf m}=(m_1,\ldots,m_{n+1})$ with $m_1\geqslant m_2\geqslant
\ldots\geqslant m_{n+1}.$ Assume that $\mathcal{X}$ has finitely
many invariant parallels. Then the number of invariant parallels of
$\mathcal{X}$ is at most $m_{n+1}$.
\end{theorem}

\begin{proof}
The same as the proof of Theorem 7 in \cite{LM1}.
\end{proof}

The contribution of our next result is twofold. On one hand we prove that a $\mathbb{E}^2-$invariant complex polynomial vector field $\X$
of degree $m=(m_1,m_2)$ with $m_1\geqslant m_2,$ can reach the maximum upper bound $m_1+1$ for the
invariant meridians in Theorem \ref{t4}. On the other hand, we show the parallelism with complex polynomial vector fields $\X$ of degree $m=(m_1,m_2)$ in the complex plane $\C^2$ which can have at most $3m_1-1$ invariant straight lines, and this bound
is reached, while if $\X$ is in $\R^2$ we only know that this bound is
in the interval $[2m_1+1,3m_1-1]$ if $m_1$ is even and
$[2m_1+2,3m_1-1]$ if $m_1$ is odd, as discussed in \cite{AGL,
CLP}.
Concretely, in statement $(b)$ of the next result we show that the maximum upper bound for invariant meridians stated by Theorem \ref{t4} cannot be reached for real polynomial quadratic vector fields $\X.$ At most $m_1$ $\X-$invariant meridians can be reached for these vector fields. In fact this is also true for cubic $\X-$invariant polynomial vector fields. We do not show these extensive computations because of their significant length.

In the following proposition we provide a complex polynomial vector
field $\X$ of degree $(2,2,2)$ on the complex ellipsoid $\mathbb{E}^2$ where
the upper bound $m_1+1=3$ for the maximum number of invariant
meridians is reached, and we prove that the real polynomial vector
field $\X$ of degree $m=(2,2,2)$ on the real ellipsoid $\mathbb{E}^2$
cannot have $m_1+1=3$ invariant meridians. Additionally we provide
real polynomial vector fields $\X$ of degree $(2,2,2)$ on the real
ellipsoid $\mathbb{E}^2$ with $m_1=2$ invariant meridians.

\begin{proposition}\label{p9}
The following statements hold.
\begin{itemize}
\item[(a)] The complex polynomial differential system
\begin{equation}\label{pp0}
\begin{array}{rl}
\dot x&= y ((i a k_2 x)/b + a_{020} y), \vspace{0.2cm}\\
\dot y&= -((b x (i a k_2 x + a_{020} b y))/a^2),\vspace{0.2cm}\\
\dot z&= 1/2 k_{001} (c^2 (-1 + x^2/a^2 + y^2/b^2) + z^2),
\end{array}
\end{equation}
on the complex ellipsoid $\mathbb{E}^2$ of degree $(2,2,2)$ has the maximum
number of complex invariant meridians $m+1=3$, namely $-i b x + a y=0$, $i b x + a y=0$ and $ia k_2 x +a_{020} b y=0.$

\item[(b)] There are no real polynomial differential systems on the
real ellipsoid $\mathbb{E}^2$ of degree $(2,2,2)$ realizing the maximum
number of real invariant meridians $m_1+1=3$. Such polynomial
differential systems have at most $2$ real invariant meridians, and
this upper bound is reached.
\end{itemize}
\end{proposition}

\begin{proof}
Consider the vector field $\X$ associated to the complex polynomial
differential system of statement $(a)$. Then
\[
\X( x^2/a^2 + y^2/b^2 + z^2/c^2-1)= k_{001} z ( x^2/a^2 + y^2/b^2 + z^2/c^2-1),
\]
so $\X$ is a polynomial vector field on the complex ellipsoid $\mathbb{E}^2$.
The three complex planes $-i b x + a y=0$, $i b x + a y=0$ and $ia k_2 x +a_{020} b y=0$ are
invariant because
\begin{equation}\label{pp1}
\begin{array}{l}
\X(-i b x + a y)= ((-i b x + a y) (a k_2 x - i a_{020} b y))/a, \vspace{0.2cm}\\
\X(i b x + a y)= -(((i b x + a y) (a k_2 x - i a_{020} b y))/a), \vspace{0.2cm}\\
\X(ia k_2 x +a_{020} b y)= ((i a k_2 x + a_{020} b y) (i a_{020} b^3 x + a^3 k_2 y))/(a^2 b).
\end{array}
\end{equation}
This completes the proof of statement $(a)$.

Now we shall prove statement $(b)$. Consider a general real
polynomial differential system of degree $(2,2,2)$ in $\R^3$, i.e.
\begin{equation}\label{eqsum}
\dot x= \sum_{i+j+k=0}^2 a_{ijk} x^i y^j z^k, \quad \dot y=
\sum_{i+j+k=0}^2 b_{ijk} x^i y^j z^k, \quad \dot z= \sum_{i+j+k=0}^2
c_{ijk} x^i y^j z^k.
\end{equation}
Imposing that $\X(x^2/a^2 + y^2/b^2 + z^2/c^2 - 1)= (k_{000}+k_{100} x+k_{010}
y+k_{001}z)(x^2/a^2 + y^2/b^2 + z^2/c^2 - 1)$ system \eqref{eqsum} becomes
\begin{equation}\label{pp2}
\begin{array}{rl}
\dot x=&-a^2 k_{100}/2 + k_{100} x^2/2 + a_{020} y^2 + a_{001} z + a_{002} z^2 + 
 y (a_{010}\\& + a_{011} z) + x (a_{110} y + a_{101} z), \vspace{0.2cm}\\
\dot y=&-b^2 k_{010}/2 + (-2 a_{110} b^2 + b^2 k_{010}) x^2/(2 a^2) + 
 k_{010} y^2/2 \\& + b_{001} z + b_{011} y z + b_{002} z^2 + 
 x (-a_{010} b^2/a^2\\&  + ((-2 a_{020} b^2 + a^2 k_{100}) y)/(2 a^2) + b_{101} z), \vspace{0.2cm}\\
\dot z=&-c^2 k_{001}/2 + (-2 a_{101} c^2 + c^2 k_{001}) x^2/(
 2 a^2) + (-2 b_{011} c^2\\& + c^2 k_{001})  y^2/(2 b^2) + k_{001} z^2/2 + 
 y (-b_{001} c^2/b^2  + (-2 b_{002} c^2\\& + b^2 k_{010}) z/(2 b^2)) + 
 x (-a_{001} c^2/a^2 + (-a_{011} b^2 c^2 - a^2 b_{101} c^2)\\& y/(
    a^2 b^2) + (-2 a_{002} c^2  + a^2 k_{100}) z/(2 a^2)).
\end{array}
\end{equation}

Now we impose that the plane $\alpha x+\beta y=0$ be invariant by the
polynomial differential system \eqref{pp2}, i.e. $\X(\alpha x+\beta y)= (k_1+k_2
x+k_3 y+k_4z)(\alpha x+\beta y)$ and we obtain $18$ polynomial
differential systems \eqref{pp2} having invariant planes of this
kind.

\begin{enumerate}
    \item The polynomial differential system
\begin{equation*}\label{pp3}
\begin{array}{rl}
\dot x=&k_{100} x^2/2 + 
 z (a_{001} + a_{101} x + a_{002} z + (-b_{011} + k_4) y \beta/\alpha) - (
 a^2 (k_{100} (b - y) (b + y) \beta\\& + 
    2 k_2 y (x \alpha + y \beta)))/(2 b^2 \beta), \vspace{0.2cm}\\
\dot y=&(a^2 k_{100} (b - y) (b + y) \alpha + 
 b^2 (-k_{100} x^2 \alpha - 
    2 z (a_{001} + a_{101} x - k_4 x+ a_{002} z) \alpha + 
    2 b_{011} y z \beta\\&  + 
    2 k_2 x (x \alpha + y \beta)))/(2 b^2 \beta), \vspace{0.2cm}\\
\dot z=&(1/(2 a^2 b^2 \alpha \beta))(-a^4 k_{100} y z \alpha^2 + 
  b^2 c^2 x \beta ((-2 a_{001} - 2 a_{101} x + k_{001} x - 
        2 a_{002} z) \alpha + 2 (b_{011}\\& - k_4) y \beta) + 
  a^2 \alpha (2 c^2 y (a_{001} + a_{101} x - k_4 x + a_{002} z) \alpha + 
     c^2 (-2 b_{011} + k_{001}) y^2 \beta\\& + 
     b^2 (-c^2 k_{001} + z (k_{100} x + k_{001} z)) \beta)),
\end{array}
\end{equation*}
having the invariant plane $\alpha x+\beta y=0$, because its
extactic polynomial is
\begin{equation*}\label{pp4}
\begin{array}{l}
-(((x \alpha + 
    y \beta) (a^2 (-b^2 k_{100} + (k_{100} - 2 k_2) y^2) \alpha + 
    b^2 (k_{100} x^2 \alpha - 2 k_2 x^2 \alpha\\\hspace{3cm}+ 
       2 z ((a_{001} + a_{101} x - k_4 x + a_{002} z) \alpha + (-b_{011} + 
             k_4) y \beta))))/(2 b^2 \alpha \beta)).
\end{array}
\end{equation*}

\item The polynomial differential system
\begin{equation*}\label{pp5}
\begin{array}{rl}
\dot x=&(1/2) (k_{100} x^2 + (
   2 z (a_{101} x \alpha + 
      a_{002} z \alpha + (-b_{011} + k_{4}) y \beta))/\alpha - (
   a^2 (k_{100} (b - y) (b + y) \beta\\& + 
      2 k_{2} y (x \alpha + y \beta)))/(b^2 \beta)), \vspace{0.2cm}\\
\dot y=&(a^2 k_{100} (b - y) (b + y) \alpha + 
 b^2 (-((k_{100} - 2 k_{2}) x^2 + 2 (a_{101} - k_{4}) x z + 
         2 a_{002} z^2) \alpha\\& + 
    2 y (k_{2} x + b_{011} z) \beta))/(2 b^2 \beta), \vspace{0.2cm}\\
\dot z=&(1/(2 a^2 b^2 \alpha \beta))(-a^4 k_{100} y z \alpha^2 + 
  b^2 c^2 x \beta ((-2 a_{101} x + k_{001} x - 2 a_{002} z) \alpha + 
     2 (b_{011} \\&- k_{4}) y \beta) + 
  a^2 \alpha (2 c^2 y (a_{101} x - k_{4} x + a_{002} z) \alpha + 
     c^2 (-2 b_{011} + k_{001}) y^2 \beta + 
     b^2 (-c^2 k_{001}\\& + z (k_{100} x + k_{001} z)) \beta)),
\end{array}
\end{equation*}
having the invariant planes $\alpha x+\beta y=0$ and $y=0$, because after setting $k_{100} = 0, a_{101} = k_{4}, a_{002} = 0, k_{2} = 0,$ its
extactic polynomial is
\begin{equation*}\label{pp6}
\begin{array}{l}
((b_{011} - k_{4}) y z (x \alpha + y \beta))/\alpha.
\end{array}
\end{equation*}

\item The polynomial differential system
\begin{equation*}\label{pp7}
\begin{array}{rl}
\dot x=&(1/2) (k_{100} x^2 + a^2 k_{100} (-1 + y^2/b^2) + 
   2 z (a_{001} + a_{101} x + a_{002} z) + (
   2 (-b_{011}\\& + k_4) y z \beta)/\alpha), \vspace{0.2cm}\\
\dot y=&(a^2 k_{100} (b - y) (b + y) \alpha + 
 b^2 (-k_{100} x^2 \alpha - 
    2 z (a_{001} + a_{101} x - k_4 x + a_{002} z) \alpha\\& + 
    2 b_{011} y z \beta))/(2 b^2 \beta), \vspace{0.2cm}\\
\dot z=&(1/(2 a^2 b^2 \alpha \beta))(-a^4 k_{100} y z \alpha^2 + 
  b^2 c^2 x \beta ((-2 a_{001} - 2 a_{101} x + k_{001} x - 
        2 a_{002} z) \alpha\\& + 2 (b_{011} - k_4) y \beta) + 
  a^2 \alpha (2 c^2 y (a_{001} + a_{101} x - k_4 x + a_{002} z) \alpha + 
     c^2 (-2 b_{011} \\&+ k_{001}) y^2 \beta + 
     b^2 (-c^2 k_{001} + z (k_{100} x + k_{001} z)) \beta)),
\end{array}
\end{equation*}
having the invariant planes $\alpha x+\beta y=0$ and $y=0$, because after setting $k_{100} = 0, a_{101} = k_{4}, a_{002} = 0, a_{001} = 0,$ its
extactic polynomial is
\begin{equation*}\label{pp61}
\begin{array}{l}
((b_{011} - k_{4}) y z (x \alpha + y \beta))/\alpha.
\end{array}
\end{equation*}

\item The polynomial differential system
\begin{equation*}\label{pp9}
\begin{array}{rl}
\dot x=&k_{100} x^2/2 - a^2 k_{100} (x y \alpha + b^2 \beta)/(
 2 b^2 \beta) + 
 z (a_{001} + a_{101} x + a_{002} z\\& + (-b_{011} + k_4) y \beta/\alpha), \vspace{0.2cm}\\
\dot y=&(a^2 k_{100} (b - y) (b + y) \alpha + 
 b^2 (-2 z (a_{001} + a_{101} x - k_4 x + a_{002} z) \alpha + 
    y (k_{100} x\\&  + 2 b_{011} z) \beta))/(2 b^2 \beta), \vspace{0.2cm}\\
\dot z=&(1/(2 a^2 b^2 \alpha \beta))(-a^4 k_{100} y z \alpha^2 + 
  b^2 c^2 x \beta ((-2 a_{001} - 2 a_{101} x + k_{001} x - 
        2 a_{002} z) \alpha\\&  + 2 (b_{011} - k_4) y \beta) + 
  a^2 \alpha (2 c^2 y (a_{001} + a_{101} x - k_4 x + a_{002} z) \alpha + 
     c^2 (-2 b_{011} \\& + k_{001}) y^2 \beta + 
     b^2 (-c^2 k_{001} + z (k_{100} x + k_{001} z)) \beta)),
\end{array}
\end{equation*}
having the invariant planes $\alpha x+\beta y=0$ and $x=0$, because after setting $k_{100} = 0, b_{011} = k_{4}, a_{002} = 0, a_{001} = 0,$ its
extactic polynomial is
\begin{equation*}\label{ppx3}
\begin{array}{l}
-((a_{101} - k_{4}) x z (x \alpha + y \beta))/\beta.
\end{array}
\end{equation*}

\item The polynomial differential system
\begin{equation*}\label{pp11}
\begin{array}{rl}
\dot x=&-(a^2 k_2 y (x \alpha + y \beta)/(b^2 \beta)) + 
 z (a_{001} + a_{101} x + a_{002} z + (-b_{011} + k_4) y \beta/\alpha), \vspace{0.2cm}\\
\dot y=&(-z (a_{001} + a_{101} x - k_4 x + a_{002} z) \alpha + b_{011} y z \beta + 
 k_2 x (x \alpha + y \beta))/\beta, \vspace{0.2cm}\\
\dot z=&(1/2) (k_{001} z^2 + 
   c^2 (-k_{001} + (-2 b_{011} + k_{001}) y^2/b^2 + (
      x (-2 a_{001} - 2 a_{101} x + k_{001} x \\&- 2 a_{002} z))/a^2 + 
      2 y (a_{001} + a_{101} x - k_4 x + a_{002} z) \alpha/(b^2 \beta) + 
      2 (b_{011} - k_4) x y \beta/(a^2 \alpha))),
\end{array}
\end{equation*}
having the invariant planes $\alpha x+\beta y=0$ and $y=0$, because after setting $k_{2} = 0, a_{101} = k_{4}, a_{002} = 0, a_{001} = 0,$ its
extactic polynomial is
\begin{equation*}\label{ppx1}
\begin{array}{l}
((b_{011} - k_{4}) y z (x \alpha + y \beta))/\alpha.
\end{array}
\end{equation*}

\item The polynomial differential system
\begin{equation*}\label{pp13}
\begin{array}{rl}
\dot x=&(1/2) (k_{100} x^2 + a^2 (-k_{100} + (k_{100} - 2 k_2) y^2/b^2 + 
   2 z (a_{001} + a_{101} x + a_{011} y + a_{002} z)), \vspace{0.2cm}\\
\dot y=&y (k_2 x + k_4 z), \vspace{0.2cm}\\
\dot z=&(1/2) (z (k_{100} x + k_{001} z) + 
   c^2 (-k_{001} + (k_{001} - 2 k_4) y^2/b^2 - (
      x (2 a_{001} + 2 a_{101} x\\& - k_{001} x + 2 a_{011} y + 2 a_{002} z))/a^2)),
\end{array}
\end{equation*}
with the invariant plane $y=0$ of multiplicity $2,$ because after setting $k_{100} = 0, k_2=0, a_{101} = k_{4}, a_{002} = 0, a_{001} = 0,$ its
extactic polynomial is
\begin{equation*}\label{ppy2}
\begin{array}{l}
-a_{011} y^2z.
\end{array}
\end{equation*}

\item The polynomial differential system
\begin{equation*}\label{pp15}
\begin{array}{rl}
\dot x=&x (a_{110} y + k_4 z), \vspace{0.2cm}\\
\dot y=&(1/2) (b^2 (-k_{010} + (-2 a_{110} + k_{010}) x^2/a^2) + k_{010} y^2 + 
   2 z (b_{001} + b_{101} x\\& + b_{011} y + b_{002} z)), \vspace{0.2cm}\\
\dot z=&(1/2) (z (k_{010} y + k_{001} z) + 
   c^2 (k_{001} - 2 k_4) x^2/a^2 + k_{001} (-1 + y^2/b^2) - 
      2 y (b_{001}\\&  + b_{101} x + b_{011} y + b_{002} z)/b^2),
\end{array}
\end{equation*}
having the invariant planes $x=0$ and $y=0$, because after setting $k_{010} = 0, b_{001} =0, k_{4}=0, b_{002} = 0, a_{110} = 0, b_{101}=0$ its
extactic polynomial is
\begin{equation*}\label{ppu3}
\begin{array}{l}
b_{011}xyz.
\end{array}
\end{equation*}

\item The polynomial differential system
\begin{equation*}\label{pp17}
\begin{array}{rl}
\dot x=&k_{100} x^2/2 + a_{101} x z + (-b_{011} + k_4) y z \beta/\alpha - (
 a^2 (k_{100} (b - y) (b + y) \beta\\& + 
    2 k_2 y (x \alpha + y \beta)))/(2 b^2 \beta), \vspace{0.2cm}\\
\dot y=&(a^2 k_{100} (b - y) (b + y) \alpha + 
 b^2 (x (-k_{100} x + 2 k_2 x - 2 a_{101} z + 2 k_4 z) \alpha\\& + 
    2 y (k_2 x + b_{011} z) \beta))/(2 b^2 \beta), \vspace{0.2cm}\\
\dot z=&(-a^4 k_{100} y z \alpha^2 + 
 b^2 c^2 x \beta ((-2 a_{101} + k_{001}) x \alpha + 
    2 (b_{011} - k_4) y \beta)\\& + 
 a^2 \alpha (2 c^2 (a_{101} - k_4) x y \alpha + 
    c^2 (-2 b_{011} + k_{001}) y^2 \beta + 
    b^2 (-c^2 k_{001}\\& + 
       z (k_{100} x + k_{001} z)) \beta))/(2 a^2 b^2 \alpha \beta),
\end{array}
\end{equation*}
having the invariant planes $\alpha x+\beta y=0$ and $y=0$, because after setting $k_{4} = 0, k_{100} =0, a_{101}=0, k_{2} = 0,$ its
extactic polynomial is
\begin{equation*}\label{ppu2}
\begin{array}{l}
b_{011}yz(\alpha x+\beta y)/\alpha.
\end{array}
\end{equation*}

\item The polynomial differential system
\begin{equation*}\label{pp19}
\begin{array}{rl}
\dot x=&(1/2) (k_{100} x^2 + a^2 k_{100} (-1 + y^2/b^2) + (
   2 z (a_{101} x \alpha + 
      a_{002} z \alpha + (-b_{011}\\& + k_4) y \beta))/\alpha), \vspace{0.2cm}\\
\dot y=&(a^2 k_{100} (b - y) (b + y) \alpha - 
 b^2 (k_{100} x^2 \alpha + 2 z (a_{101} x - k_4 x + a_{002} z) \alpha\\&  - 
    2 b_{011} y z \beta))/(2 b^2 \beta), \vspace{0.2cm}\\
\dot z=&(1/(2 a^2 b^2 \alpha \beta))(-a^4 k_{100} y z \alpha^2 + 
  b^2 c^2 x \beta ((-2 a_{101} x + k_{001} x - 2 a_{002} z) \alpha\\&  + 
     2 (b_{011} - k_4) y \beta) + 
  a^2 \alpha (2 c^2 y (a_{101} x - k_4 x + a_{002} z) \alpha + 
     c^2 (-2 b_{011} \\& + k_{001}) y^2 \beta + 
     b^2 (-c^2 k_{001} + z (k_{100} x + k_{001} z)) \beta)),
\end{array}
\end{equation*}
with the invariant planes $\alpha x+\beta y=0$ and $y=0$, because after setting $k_{100} = 0, a_{002} =0, a_{101}=k_4,$ the
extactic polynomial is
\begin{equation*}\label{ppu1}
\begin{array}{l}
(b_{011}-k_4)yz(\alpha x+\beta y)/\alpha.
\end{array}
\end{equation*}

\item The polynomial differential system
\begin{equation}\label{pp21}
\begin{array}{rl}
\dot x=&k_{100} x^2/2 + a_{101} x z + 
 a_{002} z^2 + (-b_{011} + k_4) y z \beta/\alpha - 
 a^2 k_{100} (x y \alpha\\& + b^2 \beta)/(2 b^2 \beta), \vspace{0.2cm}\\
\dot y=&(a^2 k_{100} (b - y) (b + y) \alpha + 
 b^2 (-2 z (a_{101} x - k_4 x + a_{002} z) \alpha + 
    y (k_{100} x\\&  + 2 b_{011} z) \beta))/(2 b^2 \beta), \vspace{0.2cm}\\
\dot z=&(1/(2 a^2 b^2 \alpha \beta))(-a^4 k_{100} y z \alpha^2 + 
  b^2 c^2 x \beta ((-2 a_{101} x + k_{001} x - 2 a_{002} z) \alpha\\&  + 
     2 (b_{011} - k_4) y \beta) + 
  a^2 \alpha (2 c^2 y (a_{101} x - k_4 x + a_{002} z) \alpha + 
     c^2 (-2 b_{011}\\&  + k_{001}) y^2 \beta + 
     b^2 (-c^2 k_{001} + z (k_{100} x + k_{001} z)) \beta)),
\end{array}
\end{equation}
If we set $k_{100} = 0, a_{101} = k_4$ and $a_{002} = 0,$ then the invariant planes are $\alpha x+\beta y=0$ and $y=0$ because the extactic polynomial is
\begin{equation*}\label{pp21_s}
\begin{array}{l}
((b_{011} - k_4) y z (x \alpha + y \beta))/\alpha.
\end{array}
\end{equation*}

\item The polynomial differential system
\begin{equation*}\label{pp23}
\begin{array}{rl}
\dot x=&a_{101} x z + a_{002} z^2 + (-b_{011} + k_4) y z \beta/\alpha - 
 a^2 k_2 y (x \alpha + y \beta)/(b^2 \beta), \vspace{0.2cm}\\
\dot y=&(k_2 x (x \alpha + y \beta) + 
 z (-a_{101} x \alpha + k_4 x \alpha - a_{002} z \alpha + 
    b_{011} y \beta))/\beta, \vspace{0.2cm}\\
\dot z=&(1/2) (k_{001} z^2 + 
   c^2 (-k_{001} + (-2 b_{011} + k_{001}) y^2/b^2 + 
      x (-2 a_{101} x + k_{001} x\\& - 2 a_{002} z)/a^2 + 
      2 y (a_{101} x - k_4 x + a_{002} z) \alpha/(b^2 \beta) + 
      2 (b_{011} - k_4) x y \beta/(a^2 \alpha))),
\end{array}
\end{equation*}
with the invariant planes $\alpha x+\beta y=0$ and $y=0$, because after setting $k_{2} = 0, a_{002} =0, a_{101}=k_4,$ the
extactic polynomial is
\begin{equation*}\label{ppu19}
\begin{array}{l}
(b_{011}-k_4)yz(\alpha x+\beta y)/\alpha.
\end{array}
\end{equation*}

\item The polynomial differential system
\begin{equation*}\label{pp25}
\begin{array}{rl}
\dot x=&z (a_{001} + a_{101} x + a_{002} z + (-b_{011} + k_4) y \beta/\alpha), \vspace{0.2cm}\\
\dot y=&b_{011} y z - z (a_{001} + a_{101} x - k_4 x + a_{002} z) \alpha/\beta, \vspace{0.2cm}\\
\dot z=&(1/2) (k_{001} z^2 + 
   c^2 (-k_{001} + (-2 b_{011} + k_{001}) y^2/b^2 + 
      x (-2 a_{001} - 2 a_{101} x + k_{001} x \\&- 2 a_{002} z)/a^2 + 
      2 y (a_{001} + a_{101} x - k_4 x + a_{002} z) \alpha/(b^2 \beta) + 
      2 (b_{011} - k_4) x y \beta/(a^2 \alpha))),
\end{array}
\end{equation*}
with the invariant planes $\alpha x+\beta y=0$ and $y=0$, because after setting $a_{001} = 0, a_{002} =0, a_{101}=k_4,$ the
extactic polynomial is
\begin{equation*}\label{ppu41}
\begin{array}{l}
(b_{011}-k_4)yz(\alpha x+\beta y)/\alpha.
\end{array}
\end{equation*}

\item The polynomial differential system
\begin{equation*}\label{pp27}
\begin{array}{rl}
\dot x=&(1/2) (a^2 k_{100} (-1 + y^2/b^2) + x (k_{100} x + 2 a_{101} z) + 
   2 (-b_{011} + k_4) y z \beta/\alpha), \vspace{0.2cm}\\
\dot y=&(a^2 k_{100} (b - y) (b + y) \alpha + 
 b^2 (-x (k_{100} x + 2 (a_{101} - k_4) z) \alpha + 
    2 b_{011} y z \beta))/(2 b^2 \beta), \vspace{0.2cm}\\
\dot z=&(-a^4 k_{100} y z \alpha^2 + 
 b^2 c^2 x \beta ((-2 a_{101} + k_{001}) x \alpha + 
    2 (b_{011} - k_4) y \beta)\\& + 
 a^2 \alpha (2 c^2 (a_{101} - k_4) x y \alpha + 
    c^2 (-2 b_{011} + k_{001}) y^2 \beta + 
    b^2 (-c^2 k_{001}\\& + 
       z (k_{100} x + k_{001} z)) \beta))/(2 a^2 b^2 \alpha \beta),
\end{array}
\end{equation*}
with the invariant planes $\alpha x+\beta y=0$ and $y=0$, because after setting $k_{100} = 0, a_{101}=k_4,$ the
extactic polynomial is
\begin{equation*}\label{ppu12}
\begin{array}{l}
(b_{011}-k_4)yz(\alpha x+\beta y)/\alpha.
\end{array}
\end{equation*}

\item The polynomial differential system
\begin{equation*}\label{pp29}
\begin{array}{rl}
\dot x=&(1/2) (x (k_{100} x + 2 a_{101} z) + 
   a^2 k_{100} (-1 - x y \alpha/(b^2\beta)) + 
   2 (-b_{011} + k_4) y z\beta/\alpha), \vspace{0.2cm}\\
\dot y=&(a^2 k_{100} (b - y) (b + y) \alpha + 
 b^2 (2 (-a_{101} + k_4) x z \alpha + 
    y (k_{100} x + 2 b_{011} z)\beta))/(2 b^2\beta), \vspace{0.2cm}\\
\dot z=&(-a^4 k_{100} y z \alpha^2 + 
 b^2 c^2 x\beta ((-2 a_{101} + k_{001}) x \alpha + 
    2 (b_{011} - k_4) y\beta) + 
 a^2 \alpha (2 c^2 (a_{101}\\& - k_4) x y \alpha + 
    c^2 (-2 b_{011} + k_{001}) y^2\beta + 
    b^2 (-c^2 k_{001} + 
       z (k_{100} x + k_{001} z))\beta))/(2 a^2 b^2 \alpha\beta),
\end{array}
\end{equation*}
with the invariant planes $\alpha x+\beta y=0$ and $y=0$, because after setting $k_{100} = 0, a_{101}=k_4,$ the
extactic polynomial is
\begin{equation*}\label{ppu124}
\begin{array}{l}
(b_{011}-k_4)yz(\alpha x+\beta y)/\alpha.
\end{array}
\end{equation*}

\item The polynomial differential system
\begin{equation*}\label{pp31}
\begin{array}{rl}
\dot x=&a_{101} x z + (-b_{011} + k_4) y z\beta/\alpha - 
 a^2 k_2 y (x \alpha + y\beta)/(b^2\beta), \vspace{0.2cm}\\
\dot y=&y (k_2 x + b_{011} z) + x (k_2 x + (-a_{101} + k_4) z) \alpha/\beta, \vspace{0.2cm}\\
\dot z=&(1/2) (k_{001} z^2 + 
   c^2 (-k_{001} + (-2 b_{011} + k_{001}) y^2/b^2 + 
      2 (a_{101} - k_4) x y \alpha/(b^2\beta)\\& + 
      x ((-2 a_{101} + k_{001}) x \alpha + 2 (b_{011} - k_4) y\beta)/(
      a^2 \alpha))),
\end{array}
\end{equation*}
with the invariant planes $\alpha x+\beta y=0$ and $y=0$, because after setting $k_{2} = 0, a_{101}=k_4,$ the
extactic polynomial is
\begin{equation*}\label{ppu1214}
\begin{array}{l}
(b_{011}-k_4)yz(\alpha x+\beta y)/\alpha.
\end{array}
\end{equation*}

\item The polynomial differential system
\begin{equation*}\label{pp33}
\begin{array}{rl}
\dot x=&z (a_{101} x \alpha + 
   a_{002} z \alpha + (-b_{011} + k_4) y\beta)/\alpha, \vspace{0.2cm}\\
\dot y=&z (-a_{101} x \alpha + k_4 x \alpha - a_{002} z \alpha + 
   b_{011} y\beta)/\beta, \vspace{0.2cm}\\
\dot z=&(1/2) (k_{001} z^2 + 
   c^2 (-k_{001} + (-2 b_{011} + k_{001}) y^2/b^2 + 
      x (-2 a_{101} x + k_{001} x\\& - 2 a_{002} z)/a^2 + 
      2 y (a_{101} x - k_4 x + a_{002} z) \alpha/(b^2\beta) + 
      2 (b_{011} - k_4) x y\beta/(a^2 \alpha))),
\end{array}
\end{equation*}
with the invariant planes $\alpha x+\beta y=0$ and $x=0$, because after setting $a_{002} = 0, b_{011}=k_4,$ the
extactic polynomial is
\begin{equation*}\label{ppu1234}
\begin{array}{l}
(a_{101}-k_4)xz(\alpha x+\beta y)/\beta.
\end{array}
\end{equation*}

\item The polynomial differential system
\begin{equation*}\label{pp35}
\begin{array}{rl}
\dot x=&a_{101} x z + (-b_{011} + k_4) y z\beta/\alpha, \vspace{0.2cm}\\
\dot y=&b_{011} y z + (-a_{101} + k_4) x z \alpha/\beta, \vspace{0.2cm}\\
\dot z=&(1/2) (k_{001} z^2 + 
   c^2 (-k_{001} + (-2 b_{011} + k_{001}) y^2/b^2 + 
      2 (a_{101} - k_4) x y \alpha/(b^2\beta)\\& + 
      x ((-2 a_{101} + k_{001}) x \alpha + 2 (b_{011} - k_4) y\beta)/(
      a^2 \alpha))),
\end{array}
\end{equation*}
with the invariant planes $\alpha x+\beta y=0$ and $x=0$, because after setting $b_{011}=k_4,$ the
extactic polynomial is
\begin{equation*}\label{ppu123456}
\begin{array}{l}
(a_{101}-k_4)xz(\alpha x+\beta y)/\beta.
\end{array}
\end{equation*}

\item The polynomial differential system
\begin{equation*}\label{pp37}
\begin{array}{rl}
\dot x=&k_4 x z + (-b_{011} + k_4) y z\beta/\alpha, \vspace{0.2cm}\\
\dot y=&b_{011} y z, \vspace{0.2cm}\\
\dot z=&(1/2) (k_{001} z^2 + 
   c^2 (-k_{001} + (-2 b_{011} + k_{001}) y^2/b^2 + 
      x ((k_{001} - 2 k_4) x \alpha \\&+ 2 (b_{011} - k_4) y\beta)/(
      a^2 \alpha))),
\end{array}
\end{equation*}
with the invariant planes $\alpha x+\beta y=0$ and $x=0$, because the
extactic polynomial is
\begin{equation*}\label{ppu12345}
\begin{array}{l}
(a_{101}-k_4)xz(\alpha x+\beta y)/\beta.
\end{array}
\end{equation*}
\end{enumerate}

In summary we have proved that quadratic polynomial differential
systems on the ellipsoid $\mathbb{E}^2$ can have at most $2$ invariant
meridians, and that there are such kind of systems having $2$
invariant meridians.
\end{proof}

In the next proposition we give a polynomial differential system on
the ellipsoid $\mathbb{E}^2$ where the upper bound for the maximum number of
invariant parallels is reached.

\begin{proposition}\label{p11}
Consider the polynomial differential system
\begin{equation}\label{ex2}
\begin{array}{rl}
\dot{x}&=a_{010} y , \vspace{0.2cm}\\
\dot{y}&=(1/2)  b^2  (-k_{010} + (x (-2 a_{010} + k_{010} x)/a^2) + k_{010}  y^2 + 
   2  b_{002}  z^2), \vspace{0.2cm}\\
\dot{z}&= (1/2) (-(2 b_{002} c^2)/b^2 + k_{010}) y z,
\end{array}
\end{equation}
on the ellipsoid $\mathbb{E}^2$ of degree $(2,2,1)$. For this system the
upper bound $m_3=1$ for the number of invariant parallels provided
in Theorem \ref{t5} is reached.
\end{proposition}

\begin{proof}
Let $\X$ be the vector field associated to system \eqref{ex2}. Then
$$
\X( x^2/a^2 + y^2/b^2 + z^2/c^2-1)= k_{010}y (x^2/a^2 + y^2/b^2 + z^2/c^2-1).
$$
So $\X$ is a vector field on the ellipsoid $\mathbb{E}^2$. 
Since $$\E_z= 1/2 (-((2 b_{002} c^2)/b^2) + k_{010}) y z,$$ it follows that $z=0$ is the only plane of the form $z=\mathrm{constant},$ such that it is invariant by the flow. This proves the proposition.
\end{proof}

\section*{Acknowledgements}

The first author is partially supported by the Ministerio de Ciencia, Innovaci\'on y Universidades, Agencia Estatal de Investigaci\'on grant PID2022-136613NB-100, the Ag\`encia de Gesti\'o d'Ajuts Universitaris i de Recerca grant 2021SGR00113, and the Reial Acad\`emia de Ci\`encies i Arts de Barcelona. The second author acknowledges support from the University of Texas at Dallas.


\begin{thebibliography}{99}

\bibitem {AR} {\sc R.F.S. Andrade and A. Rauh},
{\it The Lorenz model and the method of Carleman embedding}, Phys.
Lett. A {\bf 82} (1981), 276--278.

\bibitem{AGL} {\sc J.C. Art\'es, B. Gr\"unbaum and J. Llibre},
{\it On the number of invariant straight lines for polynomial
differential systems}, Pacific J. of Math. {\bf 184} (1998),
207--230.

\bibitem {King} {\sc E.H.E. Bayoumi, H.M.Soliman, F.A. El-Sheikhi, M. Shafiq, H. Awad},
{\it Robust decentralized plug and play voltage tracker of islanded
microgrids under loads and lines uncertainties by the invariant ellipsoids}, J. King Saud Univ. – Eng. Sci. {\bf 36} (2024), 409--419.

\bibitem {Naturex3} {\sc G. Babakhanova, T. Turiv, Y. Guo, M. Hendrikx, Q.H. Wei,
A.P.H.J. Schenning, D.J. Broer and O.D. Lavrentovich},
{\it Liquid crystal elastomer coatings with programmed
response of surface profile}, Nature {\bf 9} (2018), 456.

\bibitem {BGR} {\sc T.C. Bountis, A. Ramani, B. Grammaticos and B. Dorizzi},
{\it On the complete and partial integrability of non-Hamiltonian
systems}, Phys. A {\bf 128} (1984), 268--288.


\bibitem {Ca} {\sc T. Carleman},
{\it Application de la th\'eorie des \'equations int\'egrales
lin\'eaires aux syst\'emes d'\'equations diff\'erentielles non
lin\'eaires}, Acta Mathematica {\bf 59} (1932), 63--87.

\bibitem {CS} {\sc F. Cantrijn and W. Sarlet},
{\it Generalizations of Noether's theorem in classical mechanics},
SIAM Rev. {\bf 23} (1981), 467--494.

\bibitem {JMPhys} {\sc C. Castilho and H. Machado},
{\it The N-vortex problem on a symmetric ellipsoid: A
perturbation approach},
J. Math. Phys. {\bf 49} (2008), 022703.

\bibitem {CLS} {\sc J. Chavarriga, J. Llibre and J. Sotomayor},
{\it Algebraic solutions for polynomial systems with emphasis in the
quadratic case}, Expositiones Math. {\bf 15} (1997), 161--173.

\bibitem{Automatica} {\sc K. Chen, J. Lu, and C. Li},
{\it The Ellipsoidal Invariant Set of Fractional Order
Systems Subject to Actuator Saturation:
The Convex Combination Form}, J. Autom. Sinica {\bf 3} (2016), 311--319.

\bibitem{Nature2} {\sc Z. Chen, X. Ye and S.M. Fei},
{\it Quantum steerability based on
joint measurability}, Nature {\bf 7} (2017), 15822.


\bibitem {Ch} {\sc C.J. Christopher},
{\it Invariant algebraic curves and conditions for a center}, Proc.
Roy. Soc. Edinburgh {\bf 124A} (1994), 1209--1229.

\bibitem{CL1} {\sc C.J. Christopher and J. Llibre},
\emph{Algebraic aspects of integrability for polynomial systems},
Qualit. Th. Dyn. Syst. {\bf 1} (1999), 71--95.

\bibitem{CL2} {\sc C.J. Christopher and J. Llibre},
\emph{Integrability via invariant algebraic curves for planar
polynomial differential systems}, Ann. Diff. Eqs. {\bf 16} (2000),
5--19.

\bibitem{CLP} {\sc C. Christopher, J. Llibre and J.V. Pereira},
{\it Multiplicity of invariant algebraic curves in polynomial vector
fields}, Pacific J. of Math. {\bf 229} (2007), 63--117.

\bibitem{Da} {\sc G. Darboux},
{\it M\'emoire sur les \'equations diff\'erentielles alg\'ebriques
du premier ordre et du premier degr\'e (M\'elanges)}, Bull. Sci.
math. 2\`eme s\'erie {\bf 2} (1878), 60--96; 123--144; 151--200.

\bibitem{DLS} {\sc V.A. Dobrovol'skii, N.V. Lokot' and J.M. Strelcyn},
{\it Mikhail Nikolaevich Lagutinskii (1871--1915): un math\'{e}maticien
m\'{e}connu}, (French) [Mikhail Nikolaevich Lagutinskii (1871--1915): an
unrecognized mathematician] Historia Math. {\bf 25} (1998),
245--264.


\bibitem{DLA} {\sc F. Dumortier, J. Llibre and J.C. Art\'{e}s},
{\it Qualitative theory of planar differential systems},
UniversiText, Springer--Verlag, New York, 2006.

\bibitem{JPCA} {\sc G.S. Ezra and S. Wiggins},
{\it Sampling Phase Space Dividing Surfaces Constructed from Normally
Hyperbolic Invariant Manifolds (NHIMs)}, J. Phys. Chem. A {\bf 122} (2018), 8354--8362.

\bibitem{Arch} {\sc F. Fass\`o, D. Lewis},
{\it Stability Properties of the Riemann Ellipsoids}, Arch. Rat. Mech. Anal. {\bf 158} (2001), 259--292.

\bibitem{GRZ} {\sc H.J. Giacomini, C.E. Repetto and O.P. Zandron},
{\it Integrals of motion of three-dimensional non--Hamiltonian
dynamical systems}, J. Phys. A {\bf 24} (1991), 4567--4574.

\bibitem {Hi} {\sc J. Hietarinta},
{\it Direct methods for the search of the second invariant}, Phys.
Rep. {\bf 147} (1987), 87--154.

\bibitem{Nature3} {\sc X. Hu, A. Milne, B. Zhang and H. Fan},
{\it Quantum coherence of steered
states}, Nature {\bf 6} (2016), 19365.


\bibitem{I} {\sc I.D. Iliev},
{\it Perturbations of quadratic centers}, Bull. Sci. Math. {\bf 122} (1998), 107--161.

\bibitem {Jo} {\sc J.P. Jouanolou},
{\it Equations de Pfaff alg\'ebriques}, Lectures Notes in
Mathematics {\bf 708}, Springer-Verlag, New York/Berlin, 1979.

\bibitem{ND} {\sc A.A. Killin and E.N. Pivovarova},
{\it Bifurcation analysis of the problem of a “rubber” ellipsoid
of revolution rolling on a plane}, Nonlin. Dyn. {\bf 112} (2024), 15813--15832.

\bibitem{MDPI} {\sc A.A. Killin and E.N. Pivovarova},
{\it Stability of Vertical Rotations of an Axisymmetric Ellipsoid
on a Vibrating Plane}, MDPI Math. {\bf 11} (2023), 3948.

\bibitem{Science3} {\sc R. Koizumi, D. Golovaty, A. Alqarni, B.X. Li, P.J. Sternberg,
O.D. Lavrentovich},
{\it Topological transformations of a nematic drop}, Science {\bf 9} (2023), eadf3385.

\bibitem{Science1} {\sc T. Kozhukhov and T.N. Shendruk},
{\it Mesoscopic simulations of active nematics}, Science {\bf 8} (2022), eabo5788.

\bibitem{Nature1} {\sc H. Labuhn, D. Barredo, S. Ravets, S. de L\'es\'eluc, T. Macri,T. Lahaye and A. Browaeys},
{\it Tunable two-dimensional arrays of single Rydberg
atoms for realizing quantum Ising models}, Nature {\bf 534} (2016), 667--694.

\bibitem{Lax} {\sc P.D. Lax},
{\it Integrals of nonlinear equations of evolution and solitary
waves}, Commun. Pure Appl. Math. {\bf 21} (1968), 467--490.

\bibitem{Science2} {\sc Y. Li, N. Khuu, E. Prince, M. Alizadehgiashi, E. Galati,
O.D. Lavrentovich, E. Kumacheva},
{\it Nanoparticle-laden droplets of liquid crystals:
Interactive morphogenesis and dynamic assembly}, Science {\bf 5} (2019), eaav1035.

\bibitem{Ll} {\sc J. Llibre},
{\it Integrability of polynomial differential systems}, Handbook of
Differential Equations, Ordinary Differential Equations, Eds. A.
Ca\~nada, P. Drabek and A. Fonda, Elsevier, 2004, 437--533 pp.

\bibitem{LB} {\sc J. Llibre and Y. Bola\~{n}os},
{\it Rational first integrals for polynomial vector fields on
algebraic hypersurfaces of $\R^{n+1}$}, Int. J. Bifurcation and
Chaos {\bf 22} (2012), 1250270--11 pp.

\bibitem{LM} {\sc J. Llibre and J.C. Medrado},
{\it On the invariant hyperplanes for $d$--dimensional polynomial
vector fields}, J. Phys. A: Math. Gen. {\bf 40} (2007), 8385--8391.

\bibitem{LM1} {\sc J. Llibre and A.C. Murza},
{\it Darboux theory of integrability for polynomial
vector fields on $\sss^n$}, Dyn. Systems {\bf 33} (2018), 646--659.

\bibitem{LMN} {\sc J. Llibre, M. Mousavi and A. Nabavi},
{\it Limit Cycles Bifurcating from a Family of Reversible
Quadratic Centers via Averaging Theory}, Int. J. Bifurcation and
Chaos {\bf 30} (2020), 2050051--8 pp.

\bibitem{LMN1} {\sc J. Llibre, M. Mousavi and A. Nabavi},
{\it New family of centers of planar polynomial differential systems of
arbitrary even degree}, J. Dyn. Contr. Syst. {\bf 25} (2018), 619--630.


\bibitem{LP} {\sc J. Llibre and C. Pessoa},
{\it Homogeneous polynomial vector fields of degree $2$ on
the 2--dimensional ellipsoid}, Extracta Math. {\bf 21} (2006), 167--190.

\bibitem{LZ} {\sc J. Llibre and X. Zhang},
\emph{Darboux integrability of real polynomial vector fields on
regular algebraic hypersurfaces}, Rend. Circ. Mat. Palermo {\bf 51}
(2002), 109--126.

\bibitem{LZ1} {\sc J. Llibre and X. Zhang},
{\it Darboux Theory of Integrability in $\C^n$ taking into account
the multiplicity}, J. of Differential Equations {\bf 246} (2009),
541--551.

\bibitem{LZ2} {\sc J. Llibre and X. Zhang},
{\it Rational first integrals in the Darboux theory of integrability
in $\C^n$}, Bull. Sci. Math. {\bf 134} (2010), 189--195.

\bibitem{LZ3} {\sc J. Llibre and X. Zhang},
{\it On the Darboux integrability of the polynomial differential
systems}, Qualit. Th. Dyn. Syst. {\bf 11} (2012), 129--144.

\bibitem{LZ4} {\sc J. Llibre and X. Zhang},
{\it Darboux theory of integrability for polynomial vector fields in
$\mathbb R^n$ taking into account the multiplicity at infinity},
Bull. Sci. Math. {\bf 133} (2009), 765--778.

\bibitem{IFAC} {\sc J. Lou\'edec, L. Jaulin, C. Viel},
{\it Outer enclosures of nonlinear mapping with degenerate ellipsoids},
IFAC {\bf 58} (2024), 7--12.

\bibitem{JPC} {\sc G.M. McClelland and D.R. Herschbach},
{\it Symmetry Properties of Angular Correlations for Molecular Collision Complexes},
J. Phys. Chem. {\bf 83} (1979), 1445--1455.

\bibitem{Ol} {\sc P.J. Olver},
{\it Applications of Lie groups to differential equations},
Springer, New York, 1986.

\bibitem{JPCA4} {\sc H. Pan, Y. Cheng and K. Liu},
{\it Rotational Mode Specificity in $Cl + CH_4(\nu_3=1,|jNl\rangle):$ Role of Reactant’s
Vibrational Angular Momentum}, J. Phys. Chem. A {\bf 120} (2016), 4799--4804.

\bibitem{Geod} {\sc G. Panou, D. Delikarouglou and R. Korakitis},
{\it Solving the geodesics on the ellipsoid as
a boundary value problem}, J. Geod. Sci. {\bf 3} (2013), 40--47.

\bibitem{Pe} {\sc J.V. Pereira},
{\it Integrabilidade de equa\c{c}oes diferenciais no plano complexo},
Monografias del IMCA {\bf 25}, Lima, Peru, 2002.

\bibitem {Po} {\sc H. Poincar\'e},
{\it Sur l'int\'egration des \'equations diff\'erentielles du
premier ordre et du premier degr\'e I and II}, Rendiconti del
Circolo Matematico di Palermo {\bf 5} (1891), 161--191; {\bf 11}
(1897), 193--239.

\bibitem {Autom} {\sc A. Polyakov, A. Poznyak},
{\it Invariant ellipsoid method for minimization of unmatched disturbances effects
in sliding mode control}, Automatica {\bf 47} (2011), 1450--1454.

\bibitem{PS} {\sc M.J. Prelle and M.F. Singer},
{\it Elementary first integrals of differential equations}, Trans.
Amer. Math. Soc. {\bf 279} (1983), 613--636.

\bibitem {Sc1} {\sc D. Schlomiuk},
{\it Elementary first integrals and algebraic invariant curves of
differential equations}, Expositiones Math. {\bf 11} (1993),
433--454.

\bibitem {Sc2} {\sc D. Schlomiuk},
{\it Algebraic particular integrals, integrability and the problem
of the center}, Trans. Amer. Math. Soc. {\bf 338} (1993), 799--841.

\bibitem {Sc3} {\sc D. Schlomiuk},
{\it Algebraic and geometric aspects of the theory of polynomial
vector fields}, in Bifurcations and Periodic Orbits of Vector
Fields, D. Schlomiuk (ed.), 1993, 429--467 pp.

\bibitem{Si} {\sc M.F. Singer},
{\it Liouvillian first integrals of differential equations}, Trans.
Amer. Math. Soc. {\bf 333} (1992), 673--688.

\bibitem{Science4} {\sc C.M. Spaegele, M. Tamagnone, S.W.D. Lim, M. Ossiander,
M.L. Meretska, F. Capasso},
{\it Topologically protected optical polarization
singularities in four-dimensional space}, Science {\bf 9} (2023), adh0369.

\bibitem{SW} {\sc J.M. Strelcyn and S. Wojciechowski},
{\it A method of finding integrals for three--dimensional dynamical
systems}, Phys. Lett. A {\bf 133} (1988), 207--212.

\bibitem{Tsi} {\sc A.V. Tsiganov},
{\it Integrable Systems on a Sphere,
an Ellipsoid and a Hyperboloid}, Reg. Chaotic Dyn. {\bf 28} (2023), 805--821.

\bibitem{JMAA} {\sc R.M. Tudoran, A. G\^irban},
{\it On the rattleback dynamics}, J. Math. Anal. Appl. {\bf 488} (2020), 124066.

\bibitem{JPCx} {\sc C. Wulfman},
{\it Dynamical Symmetries of Schroedinger Equations and Geometrical Symmetries of Classical
Total Energy Surfaces}, J. Phys. Chem. {\bf 102} (1998), 9542--9548.

\bibitem{Naturex2} {\sc M. Yessenov, J. Free, Z. Chen, E.G. Johnson,
M.P.J. Lavery, M.A. Alonso and A.F. Abouraddy},
{\it Space-time wave packets localized in all
dimensions}, Nature {\bf 13} (2022), 4573.

\bibitem{Naturex} {\sc H. You, J. Li, J. Xu, J. Xu, T. Ning, Y. Gao and L. Li},
{\it A method for estimating magnetic
target location by employing total
field and its gradients data}, Nature {\bf 12} (2022), 17985.

\end{thebibliography}
\end{document}